\documentclass[12pt]{article}
\usepackage[a4paper, margin=2.3cm]{geometry}
\usepackage{amsmath,amssymb,amsthm}
\usepackage{color}
\usepackage{parskip}
\usepackage{url}

\newtheorem{theorem}{Theorem}[section]

\newtheorem{definition}{Definition}[section]
\newtheorem{corollary}[theorem]{Corollary}
\newtheorem{lemma}[theorem]{Lemma}

\newtheorem*{definition*}{Definition}

\DeclareMathOperator{\rank}{\mathtt{rank}}

\newcommand\be{\begin{eqnarray*}}
\newcommand\ee{\end{eqnarray*}}
\newcommand\beq{\begin{equation}}
\newcommand\eeq{\end{equation}}

\newcommand\ben{\begin{eqnarray}}
\newcommand\een{\end{eqnarray}}

\begin{document}
\title{Expanding phenomena over matrix rings}

\author{ Y. Dem\.iro\u{g}lu Karabulut\thanks{Harvey Mudd College, California. Email: {\tt ysm.demiroglu@gmail.com}}
\and D. Koh\thanks{Chungbuk National University. Email: {\tt koh131@chungbuk.ac.kr} (Corresponding Author)}\and
    T. Pham\thanks{University of California, San Diego. Email: {\tt v9pham@ucsd.edu}}
  \and
  C-Y. Shen \thanks{National Taiwan University. Email: {\tt cyshen@math.ntu.edu.tw}}
\and 
    L. A. Vinh \thanks{Vietnam Institute of Educational Sciences.
    Email: {\tt vinhla@vnu.edu.vn
}}}

\date{}
\maketitle
\date{}
\maketitle
\begin{abstract}
In this paper, we study expanding phenomena in the setting of matrix rings. More precisely, we will prove that 
\begin{itemize}
\item If $A$ is a set of $M_2(\mathbb{F}_q)$ and $|A|\gg q^{7/2}$, then we have 
\[|A(A+A)|, ~|A+AA|\gg q^4.\]
\item If $A$ is a set of $SL_2(\mathbb{F}_q)$ and $|A|\gg q^{5/2}$, then we have 
\[|A(A+A)|, ~|A+AA|\gg q^4.\]
\end{itemize}
We also obtain similar results for the cases of $A(B+C)$ and $A+BC$, where $A, B, C$ are sets in $M_2(\mathbb{F}_q)$.
\end{abstract}
\section{Introduction}
Let $\mathbb{F}_q$ be a finite field of order $q$ where $q$ is an odd prime power. Given a function $f\colon \mathbb{F}_q^d\to \mathbb{F}_q$, define 
\[f(A, \ldots, A)=\{f(a_1, \ldots, a_d)\colon a_1, \ldots, a_d\in A\},\]
the image of the set $A^d\subset \mathbb{F}_q^d$ under the function $f$. We start with the following definition of expander polynomials which can be found in \cite{hart11}. 
\bigskip
\begin{definition}
Let $f$ be a function from $\mathbb{F}_q^d$ to $\mathbb{F}_q$. 
\begin{itemize}
\item[1.] The function $f$ is called a strong expander with the exponent $\varepsilon>0$ if for all $A\subset \mathbb{F}_q$ with $|A|\gg q^{1-\varepsilon}$, one has $|f(A, \ldots, A)|\ge q-k$ for some fixed constant $k$. 
\item[2.] The function $f$ is called a moderate expander with the exponent $\varepsilon>0$ if for all $A\subset \mathbb{F}_q$ with $|A|\gg q^{1-\varepsilon}$, one has $|f(A, \ldots, A)|\gg q$.
\end{itemize}
\end{definition}
Here and throughout, $X \ll Y$ means that there exists  some absolute constant $C_1>0$ such that $X \leq C_1Y$, $X \gtrsim Y$ means $X\gg (\log Y)^{-C_2} Y$ for some absolute constant $C_2>0$, and  $X\sim Y$ means $Y\ll X\ll Y$.

Over last decades, an intensive study on expander polynomials has been made by a number of authors. Using the Kloosterman sum, Hart, Iosevich and Solymosi \cite{ha} proved that the polynomial $f(x, y, z, t)=(x-y)(z-t)$ is a strong expander with $\varepsilon=1/4$ and $k=0$. The precise statement is as follows. 
\bigskip

\begin{theorem}[\textbf{Hart-Iosevich-Solymosi}, \cite{ha}]
Let $\mathbb{F}_q$ be a finite field of order $q$ and $A$ be a set in $\mathbb{F}_q$. Suppose that $|A|\gg q^{3/4}$, then 
\[(A-A)(A-A)=\mathbb{F}_q.\]
\end{theorem}
In the case, we only want to get a positive proportion of all elements in $\mathbb{F}_q$,  Bennett, Hart, Iosevich, Pakianathan, and Rudnev \cite{bennett} proved that the threshold $q^{3/4}$ can be reduced to $q^{2/3}$ by employing Fourier techniques and tools from group action theory. This result tells us that $f(x, y, z, t)=(x-y)(z-t)$ is a moderate expander with $\varepsilon=1/3$.  

There are several families of moderate expanders with the exponent $\varepsilon=1/3$ that have been discovered over recent years. For example, $xy+zt$ by Hart and Iosevich \cite{ha2}, $x+yz$ by Shparlinski in \cite{shpas}, $x(y+z)$ and $x+(y-z)^2$ by the fifth listed author \cite{vinh}. Using methods from spectral graph theory, the fifth listed author \cite{vinh} broke the exponent $1/3$ by showing that $(x-y)^2+zt$ is a moderate expander with $\varepsilon=3/8$. 

A very general result for polynomials in two variables was given by Tao \cite{tao}. In particular, Tao \cite{tao} proved that for any polynomial $f(x, y)\in \mathbb{F}_q[x, y]$ that is not one of the forms $Q(F_1(x)+F_2(y))$ and $Q(F_1(x)F_2(y))$ for some polynomials $Q, F_1, F_2\colon \mathbb{F}_q\to \mathbb{F}_q$, we have 
\[|f(A, A)|\gg q,\]
under the assumption $|A|\gg q^{1-\frac{1}{16}}$. Therefore, such polynomials $f(x, y)$ are moderate expanders with $\varepsilon=1/16$. 

In the setting of prime fields, the third, fifth listed authors and De Zeeuw \cite{pham} showed that any quadratic polynomial in three variables $f\in \mathbb{F}_p[x,y,z]$ that depends on each variable and that does not have the form $g(h(x)+k(y)+l(z))$ for some polynomials $g, h, k, l\colon \mathbb{F}_p\to \mathbb{F}_p$ is a moderate expander with $\varepsilon=1/3$. Rudnev, Shkredov, and Stevens \cite{RSS} also proved that the function $(xy-z)(x-t)^{-1}$ is a moderate expander with $\varepsilon=17/42$ over prime fields.

Let $M_2(\mathbb{F}_q)$ be the set of two by two matrices with entries in $\mathbb{F}_q$,  $SL_2(\mathbb{F}_q)$ be the set of matrices in $M_2(\mathbb{F}_q)$ with determinant one, and $GL_2(\mathbb{F}_q)$ be the set of invertible matrices in $M_2(\mathbb{F}_q)$. Let $f\colon M_2(\mathbb{F}_q)^d\to M_2(\mathbb{F}_q)$ be a function in $d$ variables. For $A_1, \ldots, A_d\subset M_2(\mathbb{F}_q)$, we define
\[f(A_1, \ldots, A_d):=\left\lbrace f(a_1, \ldots, a_d)\colon a_i\in A_i, ~1\le i\le d\right\rbrace.\]
Similarly, in the setting of $M_2(\mathbb{F}_q)$, we have the following definition. 
\bigskip
\begin{definition}
Let $f$ be a function in $d$ variables from $M_2(\mathbb{F}_q)^d$ to $ M_2(\mathbb{F}_q).$
\begin{itemize}
\item[1.] The function $f$ is called a \textit{strong expander} over $M_2(\mathbb{F}_q)$ with the exponent $\varepsilon>0$ if for all $A\subset M_2(\mathbb{F}_q)$ with $|A|\gg q^{4-\varepsilon}$, one has $f(A, \ldots, A)\supset GL_2(\mathbb{F}_q)$.
\item[2.] The function $f$ is called a \textit{moderate expander} over $M_2(\mathbb{F}_q)$ with the exponent $\varepsilon>0$ if for all $A\subset M_2(\mathbb{F}_q)$ with $|A|\gg q^{4-\varepsilon}$, one has $|f(A, \ldots, A)|\gg q^4$. 
\item[3.] The function $f$ is called a \textit{strong expander} over $SL_2(\mathbb{F}_q)$ with the exponent $\varepsilon>0$ if for all $A\subset SL_2(\mathbb{F}_q)$ with $|A|\gg q^{3-\varepsilon}$, one has $f(A, \ldots, A)\supset GL_2(\mathbb{F}_q)$.
\item[4.] The function $f$ is called a \textit{moderate expander} over $SL_2(\mathbb{F}_q)$ with the exponent $\varepsilon>0$ if for all $A\subset SL_2(\mathbb{F}_q)$ with $|A|\gg q^{3-\varepsilon}$, one has $|f(A, \ldots, A)|\gg q^4$. 
\end{itemize}
\end{definition}
The first strong expander polynomial over $M_2(\mathbb{F}_q)$ was given by Ferguson, Hoffman, Luca, Ostafe, and Shparlinski \cite{shpa} by using an analogue of the Kloosterman over matrix rings and the approach in \cite{ha}. More precisely, they proved that for $A\subset M_2(\mathbb{F}_q)$ with $|A|\gg q^{4-\frac{1}{4}}$, we have 
\[(A-A)(A-A)\supset GL_2(\mathbb{F}_q).\]
This result implies that $(x-y)(z-t)$ is a strong expander over $M_2(\mathbb{F}_q)$ with $\varepsilon=1/4$. We note that similar results in the setting of Heisenberg group over prime fields for small sets were obtained recently by Hegyv\'{a}ri and Hennecart in \cite{HH}. Some generalizations can be found in \cite{koh2, koh3}. We refer the interested reader to \cite{chang4, soly4} and references therein for related results in the setting of $\mathbb{R}$ or $\mathbb{Z}$.

The main purpose of this paper is to provide some more families of moderate expanders over $M_2(\mathbb{F}_q)$ and $SL_2(\mathbb{F}_q)$ with exponents $7/2$ and $5/2$, respectively. The following is our first theorem. 
\bigskip
\begin{theorem}\label{nice}
Let $f(x, y, z)=x+yz$ be a function from $M_2(\mathbb{F}_q)\times SL_2(\mathbb{F}_q)^2$ to  $M_2(\mathbb{F}_q)$.
For $A\subset M_2(\mathbb{F}_q), B, C\subset SL_2(\mathbb{F}_q)$, we have 
\[|f(A, B, C)|\gg \min \left\lbrace q^4, q^3|A|, \frac{|A||B|^2|C|^2}{q^7}, \frac{|B||C|}{q}\right\rbrace.\]
\end{theorem}
As consequences, in our next two corollaries, we show that $f=x+yz$ is a moderate expander over $SL_2(\mathbb{F}_q)$ and $M_2(\mathbb{F}_q)$ with the exponents $5/2$ and $7/2$, respectively. 
\bigskip
\begin{corollary}\label{nice9}
Let $f(x, y, z)=x+yz$ be a function from $SL_2(\mathbb{F}_q)^3$ to  $M_2(\mathbb{F}_q)$. For $A\subset SL_2(\mathbb{F}_q)$ with $|A|\gg q^{5/2}$, we have 
\[|f(A, A, A)|\gg q^4.\]
\end{corollary}
\bigskip
\begin{corollary}\label{Adidaphat-09}
Let $f(x, y, z)=x+yz$ be a function from $M_2(\mathbb{F}_q)^3$ to  $M_2(\mathbb{F}_q)$. For $A\subset M_2(\mathbb{F}_q)$ with $|A|\gg q^{7/2}$, we have 
\[|f(A, A, A)|\gg q^4.\]
\end{corollary}
\bigskip
When $f(x, y, z)=x(y+z)$, we have the following result. 
\bigskip
\begin{theorem}\label{nice1}
Let $f(x, y, z)=x(y+z)$ be a function from $SL_2(\mathbb{F}_q)^2\times M_2(\mathbb{F}_q)$ to  $M_2(\mathbb{F}_q)$. For $A, B\subset SL_2(\mathbb{F}_q),  C\subset M_2(\mathbb{F}_q)$ with $|B|, |C|\gg q^2$, we have 
\[|f(A, B, C)|\gg\min\left\lbrace q^4, \frac{|A||B|^2|C|}{q^5}, \frac{|A||B|}{q}\right\rbrace.\]
\end{theorem}
\bigskip
As consequences, in our next two corollaries, we show that $f=x(y+z)$ is a moderate expander over $SL_2(\mathbb{F}_q)$ and $M_2(\mathbb{F}_q)$ with the exponents $5/2$ and $7/2$, respectively.
\bigskip
\begin{corollary}\label{nice19}
Let $f(x, y, z)=x(y+z)$ be a function from $SL_2(\mathbb{F}_q)^3$ to  $M_2(\mathbb{F}_q)$. For $A\subset SL_2(\mathbb{F}_q)$ with $|A|\gg q^{5/2}$, we have 
\[|f(A, A, A)|\gg q^4.\]
\end{corollary}
\bigskip
\begin{corollary}\label{Adidaphat-10}
Let $f(x, y, z)=x(y+z)$ be a function from $M_2(\mathbb{F}_q)^3$ to  $M_2(\mathbb{F}_q)$. For $A\subset M_2(\mathbb{F}_q)$ with $|A|\gg q^{7/2}$, we have 
\[|f(A, A, A)|\gg q^4.\]
\end{corollary}
\bigskip
In the following two theorems, we extend Theorems \ref{nice} and \ref{nice1} for arbitrary sets in $M_2(\mathbb{F}_q)$ instead of the special linear group $SL_2(\mathbb{F}_q)$. The main idea in the proofs of Theorems \ref{thm-sum2''''} and \ref{thm-sum4''''} below is to make use of the pseudo-randomness property of the \textit{sum-product digraph} which is similar to the graph constructed by Solymosi \cite{ss}. If we apply Theorems \ref{thm-sum2''''} and \ref{thm-sum4''''} for sets in $SL_2(\mathbb{F}_q)$, then the conditions are worse than those of Corollaries \ref{Adidaphat-09} and \ref{Adidaphat-10}. We will discuss about the differences between these approaches in the last section. Our next result is as follows. 
\bigskip
\begin{theorem}\label{thm-sum2''''}
Let $f(x, y, z)=x(y+z)$ be a function from $M_2(\mathbb{F}_q)^3$ to  $M_2(\mathbb{F}_q)$.
For $A, B, C\subset M_2(\mathbb{F}_q)$ with $|A||B||C|\gg q^{11}$, we have 
\[|f(A, B, C)|\gg q^4.\]
\end{theorem}
It is not hard to see that the exponent $q^{11}$ in Theorem \ref{thm-sum2''''} is sharp, since one can take $A$ as the set of zero-determinant matrices in $M_2(\mathbb{F}_q)$, $B=C=M_2(\mathbb{F}_q)$, and $|f(A, B, C)|=|A|=o(q^4)$. 
\bigskip
\begin{theorem}\label{thm-sum4''''}
Let $f(x, y, z)=x+yz$ be a function from $M_2(\mathbb{F}_q)^3$ to  $M_2(\mathbb{F}_q)$.
For $A, B, C\subset M_2(\mathbb{F}_q)$ with $|A||B||C|\gg q^{11}$, we have 
\[|f(A, B, C)|\gg q^4.\]
\end{theorem}
\bigskip
If $f$ is a polynomial in four variables of the form $f(x, y, z, t)=xy+z+t$, the following two theorems show us that $f$ is a strong expander. 
\bigskip
\begin{theorem}\label{thm-sum7}
Let $f(x, y, z, t)=xy+z+t$ be a function from $M_2(\mathbb{F}_q)^4$ to  $M_2(\mathbb{F}_q)$. Suppose that $A\subset M_2(\mathbb{F}_q)$ and $|A|\gg q^{\frac{15}{4}}$. Then we have 
\[f(A, A, A, A)=M_2(\mathbb{F}_q).\]
\end{theorem}
\bigskip
\begin{theorem}\label{thm-sum9}
Let $f(x, y, z, t)=xy+z+t$ be a function from $SL_2(\mathbb{F}_q)^2\times M_2(\mathbb{F}_q)^2$ to  $M_2(\mathbb{F}_q)$. 
For $A\subset SL_2(\mathbb{F}_q), B\subset M_2(\mathbb{F}_q)$ with $|A||B|\gg q^{\frac{27}{4}}$, then we have
\[f(A, A, B, B)=M_2(\mathbb{F}_q).\]
\end{theorem}
Our last result is devoted for an analogue of sum-product problem over the matrix ring $M_2(\mathbb{F}_q)$. 
\bigskip
\begin{theorem}\label{thm-sum1}
For $A\subset M_2(\mathbb{F}_q)$ with $|A|\gg q^{3}$, we have 
\[\max\left\lbrace |A+A|, |AA|\right\rbrace \gg \min \left\lbrace \frac{|A|^2}{q^{7/2}}, ~q^2|A|^{1/2}\right\rbrace.\]
\end{theorem}
\bigskip

As a direct consequence from Theorem \ref{thm-sum1}, we obtain the following estimates:
\begin{itemize}
\item If $|A|<q^{11/3}$, then  
\[\max\{|A+A|, |AA|\}\gg \frac{|A|^2}{q^{7/2}}.\]
\item If $|A|\ge q^{11/3}$, then 
\[\max\{|A+A|, |A A|\} \gg q^{2}|A|^{1/2}.\]
\end{itemize}
\section{Proofs of Theorems \ref{nice} and \ref{nice1}}
In the proofs of Theorems \ref{nice} and \ref{nice1}, the following two theorems play the main roles. The first theorem was given by Babai, Nikolay, and L\'{a}szl\'{o} \cite{bb}. 
\bigskip
\begin{theorem}\label{bbb}
For $A, B\subset SL_2(\mathbb{F}_q)$, we have 
\[|AB|\gg \min \left\lbrace q^3, \frac{|A||B|}{q^2}\right\rbrace.\]
\end{theorem}
\bigskip
\begin{theorem}\label{thm-main-s}
For $A\subset SL_2(\mathbb{F}_q)$, $B\subset M_2(\mathbb{F}_q)$, we have 
\[|A+B|\gg \min \left\lbrace \frac{|A|^2|B|}{q^3}, ~|A|q\right\rbrace.\]
\end{theorem}
\bigskip
To prove Theorem \ref{thm-main-s}, we first recall the expander mixing lemma for un-directed graphs. 
For an un-directed graph $G$ of order $n$, let $\lambda_1 \geq \lambda_2 \geq \ldots \geq \lambda_n$ be
the eigenvalues of its adjacency matrix. The quantity $\lambda (G) = \max
\{\lambda_2, - \lambda_n \}$ is called the second largest eigenvalue of $G$. A graph $G
= (V, E)$ is called an $(n, d, \lambda)$-graph if it is $d$-regular, has $n$
vertices, and the second largest eigenvalue of the adjacency matrix of $G$ is at most $\lambda$. 

Let $G$ be an $(n, d, \lambda)$-graph. For two  vertex subsets $B,
C \subseteq V$, let $e (B, C)$ be the number of edges between $B$ and $C$ in $G$. The following lemma gives us an estimate on the size of $e(B, C)$.
\bigskip
\begin{lemma}[Corollary 9.2.5, \cite{as}]\label{edge1}
  Let $G = (V, E)$ be an $(n, d, \lambda)$-graph. For any two sets $B, C
  \subseteq V$, we have
  \[ \left| e (B, C) - \frac{d|B | |C|}{n} \right| \leq \lambda \sqrt{|B| |C|}. \]
\end{lemma}
\bigskip
Let $\Gamma(M_2(\mathbb{F}_q), SL_2(\mathbb{F}_q))$ be the special-unit Cayley graph whose vertex set is $M_2(\mathbb{F}_q)$, and there is an edge between $a$ and $b$ if $a-b\in SL_2(\mathbb{F}_q)$. From the fact that $\det(a-b)=\det(b-a)$, the graph $\Gamma(M_2(\mathbb{F}_q), SL_2(\mathbb{F}_q))$ is an undirected graph. Using the Kloosterman sum, the first listed author \cite{Yesim} showed that $\Gamma(M_2(\mathbb{F}_q), SL_2(\mathbb{F}_q))$ is a connected graph and is an
\begin{equation}\label{x-0}(q^4, \sim q^3, 2q^{3/2})-\mbox{graph}.\end{equation}
It is interesting to note that the graph $\Gamma(M_2(\mathbb{F}_q), SL_2(\mathbb{F}_q))$ has diameter $2$. We refer the interested reader to \cite{Yesim} for more discussions. 
\paragraph{Proof of Theorem \ref{thm-main-s}:}
We consider the number of edges between $A+B$ and $B$ in the graph $\Gamma(M_2(\mathbb{F}_q), SL_2(\mathbb{F}_q))$. Let $N$ be that number.  Since $A$ is a set in $SL_2(\mathbb{F}_q)$, we always have an edge between $a+b\in A+B$ and $b\in B$ for any $a\in A, b\in B$. So $N\ge |A||B|$. 

Applying Lemma \ref{edge1} and (\ref{x-0}), we have 

\[|A||B|\le N\le \frac{|A+B||B|}{q}+2q^{3/2}\sqrt{|A+B||B|}.\]

Set $x=\sqrt{|A+B|}.$ Then we see
\[x^2|B|^{1/2}+2xq^{5/2}-q|A||B|^{1/2}\ge 0.\]
Solving this inequality, we obtain 
\[x\gg \min \left\lbrace \frac{|A||B|^{1/2}}{q^{3/2}}, |A|^{1/2}q^{1/2}\right\rbrace.\]
This completes the proof of the theorem. 
\paragraph{Proof of Theorem \ref{nice}:}
For $B, C\subset SL_2(\mathbb{F}_q)$, it follows from Theorem \ref{bbb}, we have
\begin{equation}\label{sua-1}|BC|\gg \min \left\lbrace q^3, \frac{|B||C|}{q^2}\right\rbrace.\end{equation}
Since $B$ and $C$ are subsets in $SL_2(\mathbb{F}_q)$, we have $BC$ is still a subset in $SL_2(\mathbb{F}_q)$. Applying Theorem \ref{thm-main-s}, we obtain 
\begin{equation}\label{sua2}|f(A, B, C)|=|A+BC|\gg \min\left\lbrace \frac{|BC|^2|A|}{q^3}, |BC|q\right\rbrace.\end{equation}
Combining (\ref{sua-1}) and (\ref{sua2}), we have
\[|f(A,B, C)|\gg \min \left\lbrace q^4, q^3|A|, \frac{|A||B|^2|C|^2}{q^7}, \frac{|B||C|}{q}\right\rbrace.\]
This completes the proof of the theorem. $\square$

\bigskip
To prove Theorem \ref{nice1}, we need to define an analogue of the special-unit Cayley graph.  
For $\alpha\ne 0$, let $G_\alpha$ be the graph whose the vertex set is $M_2(\mathbb{F}_q)$, and there is an edge between two vertices $a$ and $b$ if $\det(a-b)=\alpha$. It is not hard to check that $G_\alpha$ is isomorphic to $G_1=\Gamma(M_2(\mathbb{F}_q), SL_2(\mathbb{F}_q))$. Thus it is an 
\[(q^4, \sim q^3, 2q^{3/2})-\mbox{graph}.\]

We will also make use of the following lemma. 
\bigskip
\begin{lemma}\label{lmcuchuoi}
Let $i$ and $j$ be non-zero elements in $\mathbb{F}_q$. Suppose that $D_i$ and $D_j$ are two sets of matrices of determinants $i$ and $j$, respectively.
Define 
\[D_i':=\left\lbrace \begin{pmatrix}
i^{-1}a&i^{-1}b\\
c&d
\end{pmatrix}\colon \begin{pmatrix}
a&b\\
c&d
\end{pmatrix}\in D_i\right\rbrace\subset SL_2(\mathbb{F}_q), \]
and 
\[D_j':=\left\lbrace \begin{pmatrix}
j^{-1}a&b\\
j^{-1}c&d
\end{pmatrix}\colon \begin{pmatrix}
a&b\\
c&d
\end{pmatrix}\in D_j\right\rbrace\subset SL_2(\mathbb{F}_q).\]
Then we have 
\[|D_iD_j|=|D_i'D_j'|.\]
\end{lemma}
\begin{proof}
We first prove that $|D_iD_j|=|D_i'D_j|$. Indeed, let $x, y$ be two matrices in $D_i, D_j$, respectively as follows: 
\[x=\begin{pmatrix}
a&b\\
c&d
\end{pmatrix}, y=\begin{pmatrix}
e&f\\
g&h
\end{pmatrix}.
\]
We have 
\[xy=\begin{pmatrix}
ae+bg&af+bh\\
ce+gd&cf+dh
\end{pmatrix}.\]
Let $x'$ be the corresponding matrix of $x$ in $D_i'$. We have 
\[x'=\begin{pmatrix}
i^{-1}a&i^{-1}b\\
c&d
\end{pmatrix}\in SL_2(\mathbb{F}_q).\]
Observe that
\[x'y=\begin{pmatrix}
i^{-1}(ae+bg)&i^{-1}(af+bh)\\
ce+gd&cf+dh
\end{pmatrix}.\]
Since $i\ne 0$, there is a one-to-one correspondence between matrices in $D_iD_j$ and $D_i'D_j$. 

Using the same argument, we can also indicate that there is a correspondence between $D_i'D_j$ and $D_i'D_j'$. In other words, we have 
\[|D_iD_j|=|D_i'D_j|=|D_i'D_j'|.\]
\end{proof}
\paragraph{Proof of Theorem \ref{nice1}:}
We partition the set $B+C$ into $q$ subsets $D_\alpha$, $\alpha\in \mathbb{F}_q$, of matrices of determinant $\alpha$.

Since $B\subset SL_2(\mathbb{F}_q)$ and $C\subset M_2(\mathbb{F}_q)$, Theorem \ref{thm-main-s} gives us 
\[|B+C|\gg \min \left\lbrace \frac{|B|^2|C|}{q^3}, |B|q\right\rbrace>2|D_0|\sim q^3,\]
whenever $|B|, |C|\gg q^2$. Thus, without loss of generality, we assume that \[\sum_{\alpha\ne 0}|D_\alpha|\sim |B+C|.\]
Since the matrices in $AD_\alpha$ are of determinant $\alpha$, the sets $\{AD_\alpha\}_\alpha$ are distinct. Therefore, we have 
\[|A (B+C)|\gg \sum_{\alpha\ne 0}|A D_{\alpha}|.\] 

On the other hand, for each $\alpha\ne 0$, let 
\[D_\alpha':=\left\lbrace \begin{pmatrix}
\alpha^{-1}a&\alpha^{-1}b\\
c&d
\end{pmatrix}\colon \begin{pmatrix}
a&b\\
c&d
\end{pmatrix}\in D_\alpha\right\rbrace\subset SL_2(\mathbb{F}_q).\]
It is clear that $|D_\alpha'|=|D_\alpha|$. Lemma \ref{lmcuchuoi} tells us that $|AD_\alpha|=|AD_\alpha'|$. Hence, using Theorem \ref{bbb}, we get
\[|AD_{\alpha}|=|A D_{\alpha}'|\gg \min\left\lbrace q^3, \frac{|A||D_{\alpha}|}{q^2}\right\rbrace.\]
Summing over all $\alpha\ne 0$, we achieve 

\[|A(B+C)|\gg \sum_{\alpha\ne 0}|AD_\alpha| \gg \min\left\lbrace q^4, \frac{|A||B+C|}{q^2}\right\rbrace\gg \min\left\lbrace q^4, \frac{|A||B|^2|C|}{q^5}, \frac{|A||B|}{q}\right\rbrace.\]
This completes the proof of the theorem. $\square$
\section{Proofs of Corollaries \ref{Adidaphat-09} and \ref{Adidaphat-10}}
\paragraph{Proof of Corollary \ref{Adidaphat-09}:} Since $|A|\gg q^{7/2}$, without loss of generality, we may assume that $A\subset GL_2(\mathbb{F}_q)$. Thus, there exist $\beta\in \mathbb{F}_q\setminus \{0\}$ and a subset $A'\subset A$ such that  all matrices in $A'$ are of determinant $\beta$ and $|A'|\gg q^{5/2}$. 

We note that if we use the $(n, d, \lambda)$ form of the graph $G_\alpha$ from the previous section, then we are able to show that 

\[|X+Y|\gg \min \left\lbrace \frac{|X|^2|Y|}{q^3}, ~|X|q\right\rbrace,\]
for any set $X$ of matrices of determinant $\alpha$ and $Y\subset M_2(\mathbb{F}_q)$. So, with $\alpha=\beta^2$, we have
\[|A'+A'A'|\gg \min \left\lbrace \frac{|A'A'|^2|A'|}{q^3}, |A'A'|q\right\rbrace.\]
Let $A"$ be the set of corresponding matrices of determinant $1$ of matrices in $A'$ in the form of Lemma \ref{lmcuchuoi}. It follows from Lemma \ref{lmcuchuoi} and Theorem \ref{bbb} that  
\[|A'A'|=|A"A"|\gg \min\left\lbrace q^3, \frac{|A"||A"|}{q^2}\right\rbrace=\min\left\lbrace q^3, \frac{|A'|^2}{q^2}\right\rbrace.\]
Therefore, 
\[|A'+A'A'|\gg \min\left\lbrace q^4, \frac{|A'|^2}{q}, |A'|q^3, \frac{|A'|^5}{q^7}\right\rbrace\gg q^4,\]
whenever $|A'|\gg q^{5/2}$, which concludes the proof of the corollary. $\square$
\paragraph{Proof of Corollary \ref{Adidaphat-10}:}
The proof of Corollary \ref{Adidaphat-10} is almost the same with that of Corollary \ref{Adidaphat-09}, except that we follow the proof of Theorem \ref{nice1} for the set $A'$. 
$\square$
\section{Proofs of Theorems \ref{thm-sum2''''}, \ref{thm-sum4''''}, \ref{thm-sum7} and \ref{thm-sum1}}
Let $G$ be a directed graph (digraph) on $n$ vertices where the in-degree and out-degree of each vertex are both $d$. 

Let $A_G$ be the adjacency matrix of $G$, i.e., $a_{ij}=1$ if there is a directed edge from $i$ to $j$ and zero otherwise. Suppose that $\lambda_1=d, \lambda_2, \ldots, \lambda_n$ are the eigenvalues of $A_G$. These eigenvalues can be complex, so we cannot order them, but it is known that $|\lambda_i|\le d$ for all $1\le i \le n$. Define $\lambda(G):=\max_{|\lambda_i|\ne d}|\lambda_i|$. This value is called the second largest eigenvalue of $A_G$. 

We say that the $n\times n$ matrix $A$ is normal if $A^tA = AA^t$, where $A^t$ is the transpose of $A$. The graph $G$ is normal if $A_G$ is normal. There is a simple way to check whenever $G$ is normal or not. Indeed, for any two vertices $x$ and $y$, let $N^+(x,y)$ be the set of vertices $z$ such that $\overrightarrow{xz}, \overrightarrow{yz}$ are edges, and $N^-(x,y)$ be the set of vertices $z$ such that $\overrightarrow{zx}, \overrightarrow{zy}$ are  edges. By a direct computation, we have $A_G$ is normal if and only if $|N^+(x,y)| = |N^-(x,y)|$ for any two vertices $x$ and $y$. 

A digraph $G$ is called an $(n, d, \lambda)$-digraph if $G$ has $n$ vertices, the in-degree and out-degree of each vertex are both $d$, and $\lambda(G) \leq \lambda$.  Let $G$ be an $(n,d,\lambda)$-digraph. 

The following lemma is the directed version of Lemma \ref{edge1}. This was developed by Vu \cite{van}. 
\bigskip
\begin{lemma}[Vu, \cite{van}]\label{edge}
  Let $G = (V, E)$ be an $(n, d, \lambda)$-digraph. For any two sets $B, C
  \subset V$, we have
  \[ \left| e(B, C) - \frac{d}{n}|B | |C| \right| \leq \lambda \sqrt{|B| |C|},\]
  where $e(B, C)$ be the number of ordered pairs $(u, w)$ such that
$u \in B$, $w \in C$, and $\overrightarrow{uw} \in E(G)$.
\end{lemma}
To prove Theorems \ref{thm-sum2''''}, \ref{thm-sum4''''}, \ref{thm-sum7} and \ref{thm-sum1}, we need to construct the \textit{sum-product digraph} over $M_2(\mathbb{F}_q)$. Our construction is similar to that of Solymosi \cite{ss}.
\subsection{Sum-product digraph over $M_2(\mathbb{F}_q)$}
Let $G_1=(V_1, E_1)$ be the sum-product digraph over $M_2(\mathbb{F}_q)$ defined as follows:
\[V_1=M_2(\mathbb{F}_q)\times M_2(\mathbb{F}_q),\]
and there is an edge from $(A, C)$ to $(B, D)$ if 
\[A\cdot B=C+D.\]
In the following theorem, we study the $(n, d, \lambda)$ form of this digraph. 
\bigskip
\begin{theorem}\label{thm1}
The sum-product digraph $G_1$ is an 
\[(q^8, q^4, c_1q^{7/2})-\mbox{digraph}\]
for some positive constant $c_1$. 
\end{theorem}
\bigskip
Before proving this theorem, we need the following definition. 
\bigskip
\begin{definition}
Let $A$ and $B$ be two matrices in $M_2(\mathbb{F}_q)$. We say that $A$ and $B$ are \textit{equivalent} if it satisfies that for $1\le i \le 2$, the row $A_i$ of $A$ is non-zero if and only if the row $B_i$ of $B$ is  non-zero. 
\end{definition}
\begin{proof}[Proof of Theorem~\ref{thm1}]
It is obvious that the order of $G_1$ is $q^8$, because $|M_2(\mathbb{F}_q)|=q^4$ and so $|M_2(\mathbb{F}_q)\times M_2(\mathbb{F}_q)|=q^8.$
Next, one can easily prove that $G_1$ is a regular graph of in-degree  and out-degree $q^4.$ 

Let $M_1$ be the adjacency matrix of $G_1$. In the next step, we will bound the second largest eigenvalue of $G_1$. To this end, we first need to show that $G_1$ is a normal graph.
It is known that if $M_1$ is a normal matrix and $\beta$ is an eigenvalue of $M_1$, then the complex conjugate $\overline{\beta}$ is an eigenvalue of $M_1^t$. We also know that $M_1^tM_1=M_1M_1^t$ for normal matrices. Hence, we have $|\beta|^2$ is an eigenvalue of $M_1M_1^t$ and $M_1^tM_1$. In other words, in order to bound $\beta$, it is enough to bound the second largest eigenvalue of $M_1M_1^t$. 

We are now ready to show that $M_1$ is a normal. Indeed, let $(A_1, C_1)$ and $(A_2, C_2)$ be two different vertices. We now count the number of neighbors $(X, Y)$ such that there are directed edges from $(A_1, C_1)$ and $(A_2, C_2)$ to $(X, Y)$. This number is $N^+((A_1, C_1), (A_2, C_2))$.
 We first have
\begin{equation}\label{eqx}A_1X=C_1+Y, ~~A_2X=C_2+Y.\end{equation} This implies that 
\begin{equation}\label{eqy}
(A_1-A_2)X=C_1-C_2.
\end{equation}
Notice that the number of the solutions $X$ to the question \eqref{eqy} is exactly same as that of the solutions $(X,Y)$ to the system \eqref{eqx}, because if we fix a solution $X$ to $\eqref{eqy}$, then $Y$ in \eqref{eqx} is uniquely determined.
We now fall into one of the following cases:

{\bf Case 1:} If $\det(A_1-A_2)\ne 0$, then there exists a unique $X$ such that $(A_1-A_2)X=C_1-C_2$. Thus the system (\ref{eqx}) has only one solution in this case. 

{\bf Case 2:} If $\det(A_1-A_2)=0$, and $\det(C_1-C_2)\ne 0$, then the system (\ref{eqx}) has no solution. 

{\bf Case 3:} If $\det(A_1-A_2)=0$ and $\det(C_1-C_2)=0$, then we need to further consider different situations as follows:
\begin{enumerate}
\item If $\rank(A_1-A_2)=0$ and $\rank(C_1-C_2)=1$, then $A_1=A_2$. Thus it follows from \eqref{eqy} that 
\[(A_1-A_2)X=\begin{pmatrix}
		0 & 0 \\
		0 & 0 \\
		\end{pmatrix}=C_1-C_2.\] This is a contradiction since $\rank(C_1-C_2)=1$. Hence, there is no solution in this case.
\item If $\rank(A_1-A_2)=0$ and $\rank(C_1-C_2)=0$, then we have $A_1=A_2, C_1=C_2$. This contradicts with our assumption that $(A_1, C_1) \ne (A_2, C_2).$ Thus we can rule out this case. 
\item If $\rank(A_1-A_2)=1$ and $\rank(C_1-C_2)=0$, then we have $C_1=C_2$. Since $\rank(A_1-A_2)=1$, there exists at least one row of $A_1-A-2$ which is different from $(0,0)$. Hence, without loss of generality we may assume that
 \[A_1-A_2=\begin{pmatrix}
a&b\\
\alpha a&\alpha b\\
\end{pmatrix} \]
where $(a,b) \neq (0,0)$ and $\alpha\in \mathbb{F}_q$. Let $X=\begin{pmatrix}
x&y\\
z&t\\
\end{pmatrix}$. The system (\ref{eqy}) gives us 
\[ax+bz=0, ay+bt=0.\]
Since $(a, b)\ne (0,0)$, we have 
\[z=-b^{-1}ax \quad \text{or} \quad x=-a^{-1}bz.\] Hence, once we choose $x\in \mathbb F_q$, we have a unique $z$ or vice versa. A similar relation exists between $y$ and $t$ as well. That means the number of solutions to (\ref{eqx}) is $q^2,$ because  $Y$ is uniquely determined whenever we fix a solution $X$ to \eqref{eqy}. 
\item Assume that $\rank(A_1-A_2)=1$, $\rank(C_1-C_2)=1$,  $A_1-A_2$ and $C_1-C_2$ are equivalent, and $\alpha=\beta$, where 
\[A_1-A_2=\begin{pmatrix}
a&b\\
\alpha a&\alpha b\
\end{pmatrix}, C_1-C_2=\begin{pmatrix}
u&v\\
\beta u&\beta v\
\end{pmatrix}.\]

Without loss of generality we can suppose that $(a, b)\ne (0, 0)$ and $(u, v)\ne (0, 0)$. If $X$ is a solution of (\ref{eqy}), then we have 
\[ax+bz=u, ~ay+bt=v.\]
We can write $x$ in terms of $z$ or vice versa, since $(a, b)\ne (0, 0)$. Similarly we have $q$ many solutions for $(y, t)$. So the number of solutions $(X, Y)$ to the system (\ref{eqx}) is $q^2$.
\item If $\rank(A_1-A_2)=1$, $\rank(C_1-C_2)=1$, and either $A_1-A_2$ and $C_1-C_2$ are not equivalent or $\alpha\ne \beta$, where 
\[A_1-A_2=\begin{pmatrix}
a&b\\
\alpha a&\alpha b\
\end{pmatrix}, C_1-C_2=\begin{pmatrix}
u&v\\
\beta u&\beta v\
\end{pmatrix},\]
then it is not hard to see that there is no solution to (\ref{eqy}). Thus, in this case, there does not exist any solution to the system (\ref{eqx}).
\end{enumerate}

Since the same argument works for the case of $N^-((A_1, C_1), (A_2, C_2))$, we obtain the same value for $N^-((A_1, C_1), (A_2, C_2))$. In short, $M_1$ is normal. 

As we discussed above, in order to bound the second largest eigenvalue of $M_1$, it is enough to bound the second largest eigenvalue of $M_1M_1^t$. Based on previous calculations, we have
\[M_1M_1^t=(q^4-1)I+J-E_{11}-E_{12}+(q^2-1)E_{13}+(q^2-1)E_{14}-E_{15} \]
where $I$ is the identity matrix, $J$ denotes the all-one matrix, and $E_{1i}$'s, $1 \leq i \leq 5,$  are defined as follows.
$E_{11}$ is the adjacency matrix of the graph $\mathcal{G}_{11}$ defined as follows:
\[V_{11}=M_2(\mathbb{F}_q)\times M_2(\mathbb{F}_q), \]
and there is an edge between $(A_1, C_1)$ and $(A_2, C_2)$ if 
\[\det(A_1-A_2)=0, ~\det(C_1-C_2)\ne 0.\]

$E_{12}$ is the adjacency matrix of the graph $\mathcal{G}_{12}$ defined as follows:
\[V_{12}=M_2(\mathbb{F}_q)\times M_2(\mathbb{F}_q), \]
and there is an edge between $(A_1, C_1)$ and $(A_2, C_2)$ if 
\[\rank(A_1-A_2)=0, ~\rank(C_1-C_2)=1.\]

$E_{13}$ is the adjacency matrix of the graph $\mathcal{G}_{13}$ defined as follows:
\[V_{13}=M_2(\mathbb{F}_q)\times M_2(\mathbb{F}_q), \]
and there is an edge between $(A_1, C_1)$ and $(A_2, C_2)$ if 
\[\rank(A_1-A_2)=1, ~\rank(C_1-C_2)=0.\]

$E_{14}$  is the adjacency matrix of the graph $\mathcal{G}_{14}$ defined as follows:
\[V_{14}=M_2(\mathbb{F}_q)\times M_2(\mathbb{F}_q), \]
and there is an edge between $(A_1, C_1)$ and $(A_2, C_2)$ if $\rank(A_1-A_2)=1, ~\rank(C_1-C_2)=1$, $A_1-A_2$ and $C_1-C_2$ are equivalent, and $\alpha=\beta$ where 
\[A_1-A_2=\begin{pmatrix}
a&b\\
\alpha a&\alpha b\
\end{pmatrix}, C_1-C_2=\begin{pmatrix}
u&v\\
\beta u& \beta v\
\end{pmatrix}\] OR \[A_1-A_2=\begin{pmatrix}
\alpha a& \alpha b\\
 a & b\
\end{pmatrix}, C_1-C_2=\begin{pmatrix}
\beta u & \beta v\\
u & v\
\end{pmatrix}\] for some $(a,b)\ne(0,0)$ and $\alpha, \beta \in \mathbb{F}_q$.

$E_{15}$  is the adjacency matrix of the graph $\mathcal{G}_{15}$ defined as follows:
\[V_{15}=M_2(\mathbb{F}_q)\times M_2(\mathbb{F}_q), \]
and there is an edge between $(A_1, C_1)$ and $(A_2, C_2)$ if $\rank(A_1-A_2)=1, ~\rank(C_1-C_2)=1$, and either $A_1-A_2$ and $C_1-C_2$ are not equivalent, or $\alpha\ne \beta$, where 
\[A_1-A_2=\begin{pmatrix}
a&b\\
\alpha a&\alpha b\
\end{pmatrix}, C_1-C_2=\begin{pmatrix}
u&v\\
\beta u&\beta v\
\end{pmatrix}.\]

Suppose $\lambda_2$ is the second largest eigenvalue of $M_{1}$ and $\vec{v}_{2}$ is the corresponding eigenvector. Since $G_1$ is a regular graph, we have $J\cdot \vec{v}_{2}=0$. (Indeed, since $G_1$ is regular, it always has $(1,1,\cdots,1)$ as an eigenvector with eigenvalue being its regular-degree. Moreover, since the graph $G_1$ is connected, this eigenvalue has multiplicity one. Thus any other eigenvectors will be orthogonal to $(1,1,\cdots,1)$ which in turns gives us $J\cdot \vec{v}_{2}=0$). Since $M_1M_1^tv_2=|\lambda_2|^2v_2$, we get
\[|\lambda_2|^2 \vec{v}_{2}=(q^4-1)\vec{v}_{2}+\left(-E_{11}-E_{12}+(q^2-1)E_{13}+(q^2-1)E_{14}-E_{15}\right)\vec{v}_{2}.\]

Thus $\vec{v}_{2}$ is an eigenvector of 
\[(q^4-1)I+\left(-E_{11}-E_{12}+(q^2-1)E_{13}+(q^2-1)E_{14}-E_{15}\right),\]
and $|\lambda_2|^2$ is the corresponding eigenvalue. 

One can easily check that for any $1 \leq i \leq 5$, the graphs $\mathcal{G}_{1i}$ are $k_i$-regular for some $k_i$. It is not hard to check that $k_i \ll q^5$ for $2\le i\le 5$ and $k_1\ll q^7$.

Since eigenvalues of a sum of matrices are bounded by the sum of the largest eigenvalues
of the summands, we obtain
\[|\lambda_2| \ll q^{7/2} \]
which completes the proof of the theorem. 
\end{proof}
\paragraph{Proof of Theorem \ref{thm-sum2''''}:}
Since $|A||B||C|\gg q^{11}$, we have $|A|\gg q^3$. On the other hand, the number of matrices in $M_2(\mathbb{F}_q)$ with zero-determinant is $q^3+q^2-q.$ Thus, without loss of generality, we may assume that $A\subset GL_2(\mathbb{F}_q)$. 
Define
\[U:=\left\lbrace (a^{-1}, b)\colon a\in A, b\in B\right\rbrace, ~V:= A (B+C)\times C,\]
as subsets of vertices in the sum-product digraph $G_1$. It is clear that $|U|=|A||B|$ and $|V|=|C||A(B+C)|$. 

For each vertex $(a^{-1}, b)$ in $U$, it has at least $|C|$ neighbors $(a(b+c), c)\in V$. Therefore, the number of edges between $U$ and $V$ in the digraph $G_1$ is at least $|A||B||C|$. On the other hand, it follows from Theorem \ref{thm1} and Lemma \ref{edge} that 
\[e(U, V)\ll \frac{|A||B||C||A(B+C)|}{q^4}+q^{7/2}(|A||B||C|)^{1/2}\sqrt{|A(B+C)|}.\]
So, 
\[|A||B||C|\ll \frac{|A||B||C||A(B+C)|}{q^4}+q^{7/2}(|A||B||C|)^{1/2}\sqrt{|A(B+C)|}).\]
Solving this inequality, we get
\[|A(B+C)|\gg \min\left\lbrace \frac{|A||B||C|}{q^7}, q^4\right\rbrace,\]
and the theorem follows. $\square$
\paragraph{Proof of Theorem \ref{thm-sum4''''}:}
Define 
\[U=\{(b, -a )\colon a\in A, b\in B\}, ~V=C\times (A+BC),\]
as subsets of vertices in the sum-product digraph $G_1$. It is clear that $|U|= |A||B|$ and $|V|=|C||A+BC|$. 

One can check that each vertex $(b, -a)$ in $U$ has at least $|C|$ neighbors $(c, a+b\cdot c)\in V$.

This implies that $e(U, V)\ge |A||B||C|$. On the other hand, it follows from Theorem \ref{thm1} and Lemma \ref{edge} that 
\[e(U, V)\ll \frac{|A||B||C|x^2}{q^4}+q^{7/2}(|A||B||C|)^{1/2}x,\]
where $x=\sqrt{|A+BC|}$.

So \[|A||B||C|\ll \frac{|A||B||C|x^2}{q^4}+q^{7/2}\sqrt{|A||B||C|}\,x.\]
Solving this inequality, we obtain 
\[|A+BC| \gg \min\left\lbrace \frac{|A||B||C|}{q^7}, q^4\right\rbrace,\]
which completes the proof of the theorem. $\square$
\paragraph{Proof of Theorem \ref{thm-sum7}:}
Let $M$ be an arbitrary matrix in $M_2(\mathbb{F}_q)$. We will show that there exist matrices $a_1, a_2, a_3, a_4\in A$ such that 
\[a_1\cdot a_2+a_3+ a_4=M.\]
Define 
\[U:=\{(a_1, -a_3+M)\colon a_1, a_3\in A\},\]
and 
\[V:=\{(a_2, -a_4)\colon a_2, a_4\in A\},\]
as vertex subsets in the sum-product digraph $G_1$ over $M_2(\mathbb{F}_q)$.

It is clear that if there is an edge between $U$ and $V$, then 
there exist matrices $a_1, a_2, a_3, a_4\in A$ such that 
\[a_1\cdot a_2+a_3+ a_4=M.\]
It follows from Theorem \ref{thm1} and Lemma \ref{edge} that 
\[\left\vert e(U, V)-\frac{|U||V|}{q^4}\right\vert\ll q^{7/2} \sqrt{|U||V|}.\]
Since $|U|=|V|=|A|^2$, we have 
\[e(U, V)>0,\]
under the condition $|A|\gg q^{15/4}$. $\square$
\paragraph{Proof of Theorem \ref{thm-sum1}:}
Since $|A|\gg q^3$, we assume that $A\subset GL_2(\mathbb{F}_q)$.  We define:
\[U:=(A+A)\times (AA), ~V:=\left\lbrace (a, a\cdot b )\colon a, b\in A\right\rbrace,\]
as subsets of vertices in the sum-product graph $G_1$. 

It is clear that 
\[|U|=|AA||A+A|, ~|V|=|A|^2.\]
Moreover, for each vertex $(a, a\cdot b)\in V$, it has at least $|A|$ neighbors $(c+b, a\cdot c)$ in $U$. Thus the number of edges between $U$ and $V$ is at least $|A|^3$. 

On the other hand, it follows from Theorem \ref{thm1} and Lemma \ref{edge} that 
\[e(U, V)\ll \frac{|U||V|}{q^4}+q^{7/2}\sqrt{|U||V|}.\]
Hence, we obtain 
\[|A|^3\ll \frac{|A+A||A A||A|^2}{q^4}+q^{7/2}|A|\sqrt{|A+A||AA|}.\]
Set $x=\sqrt{|A+A||A A|}.$ It follows that
\[|A|x^2+q^{15/2}x-q^4|A|^2\ge 0.\]
Solving this inequality, we get
\[x\ge\frac{-q^{15/2}+\sqrt{q^{15}+4q^4|A|^3}}{2|A|},\]
which implies
\[x\gg \min \left\lbrace \frac{|A|^2}{q^{7/2}}, ~q^2|A|^{1/2}\right\rbrace.\]
On the other hand, we observe that 
\[\max\{|A+A|, |AA|\}\ge x,\]
which completes the proof of Theorem \ref{thm-sum1}. $\square$
\section{Proof of Theorem \ref{thm-sum9}}
The idea to prove Theorem \ref{thm-sum9} is the same with that of Theorem \ref{thm-sum7}, except that we will use the \textit{sum-product digraph over} $SL_2(\mathbb{F}_q)$ which is constructed as follows.
\subsection{Sum-product digraph over $SL_2(\mathbb{F}_q)$}
Let $G_2=(V_2, E_2)$ be the sum-product digraph over $SL_2(\mathbb{F}_q)$ defined as follows. The  vertex set $V_2$ is  $SL_2(\mathbb{F}_q)\times M_2(\mathbb{F}_q),$ and
 there is an edge from $(A, C)$ to $(B, D)$ if and only if $A\cdot B=C+D$. We have $|V_2|=|SL_2(\mathbb{F}_q)|\cdot |M_2(\mathbb{F}_q)|\sim q^7$.

For each vertex $(A, C)$, we now count the number of vertices $(B, D)$ such that there is an edge from $(A, C)$ to $(B, D)$, i.e., $A\cdot B=C+D$. It is clear that for each $B\in SL_2(\mathbb{F}_q)$, the matrix $D$ is uniquely determined. This means that the out-degree of each vertex in $G_2$ is $d=|SL_2(\mathbb{F}_q)|\sim q^3$. The same computation also holds for the in-degree of each vertex. In short, $G_2$ is a regular digraph of degree $d=|SL_2(\mathbb{F}_q)|\sim q^3$. In the following theorem, we give the $(n, d, \lambda)$ form of the digraph $G_2$. 
\bigskip
\begin{theorem}\label{thm2}
We have $G_2$ is an 
\[\left(|SL_2(\mathbb{F}_q)|\cdot |M_2(\mathbb{F}_q)|, |SL_2(\mathbb{F}_q)|, cq^{11/4}\right)-\mbox{digraph},\]
for some positive constant $c$.
\end{theorem}
\paragraph{Proof of Theorem \ref{thm2}:}
As observed above, the degree of $G_2$ is $|SL_2(\mathbb{F}_q)|\cdot |M_2(\mathbb{F}_q)|\sim q^7$, and 
$G_2$ is a regular digraph of in-degree and out-degree $~|SL_2(\mathbb{F}_q)|.$ Hence it suffices to  show that the second largest eigenvalue of the adjacency matrix of $G_2$ is at most $cq^{11/4}$. As in the proof of Theorem \ref{thm1}, we will see that $G_2$ is normal in the next step. Thus it is enough to bound the second largest eigenvalue of the matrix $M_2M_2^t$, where $M_2$ is the adjacency matrix of $G_2$. 

Suppose $(A_1, C_1) \ne (A_2, C_2)$ are two vertices of $G_2.$ We now count the number of vertices $(X, Y)\in SL_2(\mathbb{F}_q)\times M_2(\mathbb{F}_q)$ such that there are edges from $(A_1, C_1)$ and  $(A_2, C_2)$ to $(X, Y)$. This is equivalent with the following system
\begin{equation} \label{ThangSL:eq1} A_1X=C_1+Y, ~~A_2X=C_2+Y.\end{equation}
This implies that 
\begin{equation} \label{ThangSL:eq2}(A_1-A_2)X=C_1-C_2.\end{equation}
We now consider the following cases:

{\bf Case $1$:} $\det(A_1-A_2)=\det(C_1-C_2)\ne 0$. In this case $X$ is uniquely determined, and so is $Y$. 
Thus there is exactly one solution $(X, Y)$ to the system \eqref{ThangSL:eq1}.

 {\bf Case $2$:} $\det(A_1-A_2)\ne 0$, $\det(C_1-C_2)\ne 0$, and $\det(A_1-A_2)\ne \det(C_1-C_2)$. In this case, there is no solution $(X, Y)$ with $\det(X)=1$. 
 
 {\bf Case $3$:} $\det(A_1-A_2)=0$, $\det(C_1-C_2)\ne 0$. In this case, there is no solution $(X, Y)$ with $\det(X)=1$. 
  
  {\bf Case $4$:} $\det(A_1-A_2)\ne 0$, $\det(C_1-C_2)=0$. In this case, there is no solution $(X, Y)$ with $\det(X)=1$. 
  
  {\bf Case $5$:} $\det(A_1-A_2)=\det(C_1-C_2)=0$. In this case, we further consider the following:
  
  \begin{enumerate}
  \item[5.1.] $\rank(A_1-A_2)=0$ and $\rank(C_1-C_2)=1$. In this case we have $A_1=A_2$ and $C_1\ne C_2$. Thus there is no solution $(X, Y)$ such that $(A_1-A_2)X=C_1-C_2$. 
  \item[5.2.] $\rank(A_1-A_2)=1$ and $\rank(C_1-C_2)=0$. We now count the number of $X\in SL_2(\mathbb{F}_q)$ such that $(A_1-A_2)X=\mathbf{0}$. Without loss of generality, we assume that 
  \[A_1-A_2=\begin{pmatrix}
  u&v\\
  \alpha u&\alpha v\\
  \end{pmatrix},~~X=\begin{pmatrix}
  x_1&y_1\\
  x_2&y_2\\
  \end{pmatrix}\]
  for some $\alpha\in \mathbb{F}_q$, $u\ne 0$, and $x_1y_2-x_2y_1=1$. The equation $(A_1-A_2)X=\mathbf{0}$ gives us 
  \[ux_1+vx_2=0, ~uy_1+vy_2=0.\]
Thus we obtain 
\[u(x_1, y_1)^T+v(x_2, y_2)^T=(0, 0)^T.\]
This leads to the fact that two vectors $(x_1, y_1)$ and $(x_2, y_2)$ are linearly dependent. Hence, $\det(X)=0$, which in turn shows there is no solution $(X, Y)$ in this case. 
\item[5.3.] $\rank(A_1-A_2)=1=\rank(C_1-C_2)$ and $\alpha=\beta$, where 
 \[A_1-A_2=\begin{pmatrix}
  u&v\\
  \alpha u&\alpha v\\
  \end{pmatrix} ~~C_1-C_2=\begin{pmatrix}
  a&b\\
  \beta a&\beta b\\
  \end{pmatrix}.\]
Suppose that 
  \[
  X=\begin{pmatrix}
  x_1&y_1\\
  x_2&y_2\\
  \end{pmatrix},\]
with $x_1y_2-y_1x_2=1$.
We shall show that the number of solutions $(X, Y)$ to the system \eqref{ThangSL:eq1} is exactly $q.$  
To this end, it suffices to find the number of solutions $X$ to the equation \eqref{ThangSL:eq2}. Since $\alpha=\beta$,
the system \eqref{ThangSL:eq2} is equivalent to 
\begin{equation}\label{eq1}
ux_1+vx_2=a, ~~uy_1+vy_2=b.\end{equation}

Since $\rank(C_1-C_2)=1$, we may assume that $a\ne 0$. 

If $v=0$, then we have $u\ne 0$. Therefore, we get $x_1=a/u$ and $y_1=b/u$. 

On the other hand, we also have $x_1y_2-y_1x_2=1$. Since $x_1\ne 0$, for each choice of $x_2$ in $\mathbb{F}_q$, $y_2$ is uniquely determined. In other words, if $v=0$, then the number of solutions $X$ to the equation $(A_1-A_2)X=(C_1-C_2)$ is the number of matrices of the following form
\[
\begin{pmatrix}
\frac{a}{u}&\frac{b}{u}\\
x_2&y_2\\
\end{pmatrix}.\]
Therefore, in this case, we have $q$ solutions $X$ as expected. 

If $u=0$, we can repeat the same argument as above. 

If $u\ne 0$ and $v\ne 0$, then we choose $x_1$ arbitrarily with the following two cases. 

Suppose $x_1=0.$ Since the solution $X$ belongs to $SL_2(\mathbb F_q)$, $X$ is of the following form 
\[X=\begin{pmatrix}
0&\frac{-1}{x_2}\\
x_2&y_2\\
\end{pmatrix}.\]
From this observation and the system \eqref{eq1}, we see that
$$ x_1=0,~ x_2=\frac{a}{v}, ~y_1=-\frac{v}{a} ~~\mbox{and}~~ y_2=(b-uy_1)/v.$$ 
Thus there is exactly one solution $X$ such that $x_1=0.$

Next, suppose $x_1\ne 0.$ Since $\det(X)=1$, we have 
\[y_2=\frac{1+x_2y_1}{x_1}.\]
Substituting this to the system \eqref{eq1}, we obtain 
\[y_1=\frac{bx_1-v}{a}, ~~x_2=\frac{a-ux_1}{v}.\]
This implies that for each fixed $x_1\ne0$, there is a unique solution $X$ to the equation $(A_1-A_2)X=(C_1-C_2).$
In conclusion, for each $x_1\in \mathbb{F}_q$, there is only a matrix $X\in SL_2(\mathbb{F}_q)$ such that 
\[(A_1-A_2)X=(C_1-C_2).\] 
Thus, in this case the number of solutions $(X, Y)$ to the system \eqref{ThangSL:eq1} is $q$ as desired. 

\item[5.4.]$\rank(A_1-A_2)=1=\rank(C_1-C_2)$ and $\alpha\ne \beta$, where 
 \[A_1-A_2=\begin{pmatrix}
  u&v\\
  \alpha u&\alpha v\\
  \end{pmatrix} ~~C_1-C_2=\begin{pmatrix}
  a&b\\
  \beta a&\beta b\\
  \end{pmatrix}.\]
  
It follows from the argument in the case $5.3$ that there is no solution $(X, Y)$.   
\end{enumerate}
We now can express $M_2M_2^t$ as follows.
\begin{equation}\label{eq33}M_2M_2^t=(|SL_2(\mathbb{F}_1)|-1)I+J-M_1+\sum_{i\in \mathbb{F}_q\setminus \{0\}}M_{2i}-M_3-M_4-M_5-M_6+(q-1)M_7-M_8,\end{equation}
where 

$I$ is the identity matrix, 

$J$ is the all-one matrix, 

$M_1$ is the adjacency matrix of the graph $G'_{1}=(V'_{1}, E'_{1})$ defined as follows: 
\[V'_{1}=SL_2(\mathbb{F}_q)\times M_2(\mathbb{F}_q), \]
and there is an edge between $(A, C)$ and $(B, D)$ if $\det(A-B)\ne 0$ and $\det(C-D)\ne 0$,

$M_{2i}$ is the adjacency matrix of the graph $G_{2i}=(V_{2i}, E_{2i})$ defined as follows: 
\[V_{2i}=SL_2(\mathbb{F}_q)\times M_2(\mathbb{F}_q), \]
and there is an edge between $(A, C)$ and $(B, D)$ if $\det(A-B)=\det(C-D)=i$,

$M_3$ is the adjacency matrix of the graph $G_3=(V_3, E_3)$ defined as follows: 
\[V_3=SL_2(\mathbb{F}_q)\times M_2(\mathbb{F}_q), \]
and there is an edge between $(A, C)$ and $(B, D)$ if $\det(A-B)=0, \det(C-D)\ne 0$,

$M_4$ is the adjacency matrix of the graph $G_4=(V_4, E_4)$ defined as follows: 
\[V_4=SL_2(\mathbb{F}_q)\times M_2(\mathbb{F}_q), \]
and there is an edge between $(A, C)$ and $(B, D)$ if $\det(A-B)\ne 0, \det(C-D)= 0$,

$M_5$ is the adjacency matrix of the graph $G_5=(V_5, E_5)$ defined as follows: 
\[V_5=SL_2(\mathbb{F}_q)\times M_2(\mathbb{F}_q), \]
and there is an edge between $(A, C)$ and $(B, D)$ if $\rank(A-B)=0, \rank(C-D)=1$,

$M_6$ is the adjacency matrix of the graph $G_6=(V_6, E_6)$ defined as follows: 
\[V_6=SL_2(\mathbb{F}_q)\times M_2(\mathbb{F}_q), \]
and there is an edge between $(A, C)$ and $(B, D)$ if $\rank(A-B)=1, \rank(C-D)=0$,

$M_7$ is the adjacency matrix of the graph $G_7=(V_7, E_7)$ defined as follows: 
\[V_7=SL_2(\mathbb{F}_q)\times M_2(\mathbb{F}_q), \]
and there is an edge between $(A, C)$ and $(B, D)$ if $\rank(A-B)=1=\rank(C-D)$ and $\alpha=\beta$,

$M_8$ is the adjacency matrix of the graph $G_8=(V_8, E_8)$ defined as follows: 
\[V_8=SL_2(\mathbb{F}_q)\times M_2(\mathbb{F}_q), \]
and there is an edge between $(A, C)$ and $(B, D)$ if $\rank(A-B)=1=\rank(C-D)$ and $\alpha\ne \beta$.

In order to bound the second largest eigenvalue of $M$, we will need to bound the second largest  eigenvalues of $M_1,\ldots, M_8$. 
\subsubsection{The eigenvalues of $M_1$}
Let $G_{11}=(V_{11}, E_{11})$ be a Cayley graph defined as follows:

\[V_{11}=SL_2(\mathbb{F}_q), (A, B)\in E_{11}\Leftrightarrow A-B\in GL_2(\mathbb{F}_q).\]
\begin{lemma}
The graph $G_{11}$ is a regular graph of degree $d_{11}=q^3-q^2-q\sim q^3$.
\end{lemma}
\begin{proof}
For \[A=\begin{pmatrix}
a&b\\
c&d\\
\end{pmatrix}\in SL_2(\mathbb{F}_q),\]
with $d=(1+bc)/a$ (assume that $a\ne 0$), we now count the number of matrices 
\[X=\begin{pmatrix}
x&y\\
z&t\\
\end{pmatrix}\in SL_2(\mathbb{F}_q),\]
such that $\det(A-X)=0$ and $A\ne X.$ 

Since $\det(A-X)=0$ and $A\ne X$, we have $\rank(A-X)=1$. This means that there exists $\lambda\in \mathbb{F}_q$ such that 

\begin{equation}\label{eq2}c-z=\lambda(a-x), ~~\frac{1+bc}{a}-t=\lambda (b-y),\end{equation}

or 

\[a=x,~\mbox{and}~y=b.\]

We first count the matrices $X$ with $x=0$. 

Since $x=0$, we have $y$ needs to be non-zero, and $X$ is of the following form 
\[X=\begin{pmatrix}
0&y\\
-\frac{1}{y}&t\\
\end{pmatrix}.\]
Since $\det(A-X)=0$, we have 
\begin{equation}\label{eq3}(1+bc)-at=(b-y)\left(c+\frac{1}{y}\right).\end{equation}
Thus for each non-zero $y$ in $\mathbb{F}_q$, $t$ is  uniquely determined. In short, there are $q-1$ matrices with $x=0$. 

We now count matrices $X$ with $x\ne 0$.

\begin{itemize}
\item Suppose we are in the first case, i.e. there exists such a $\lambda$. If we choose $\lambda=c/a$, then we have $z=\lambda x$.  

With these parameters, the second equation of the system (\ref{eq2}) tells us that $x=a$, and $z=c$. For arbitrary $y\in \mathbb{F}_q\setminus\{b\}$, we get the desired matrices $X$. In short, the number of matrices $X\in SL_2(\mathbb{F}_q)$ such that $\det(A-X)=0$ and $\lambda=c/a$ is $q-1$. 

If $\lambda\ne c/a$, then we need to choose $x\ne a$, since otherwise $X=A$.  Moreover, for $x\in \mathbb{F}_q\setminus\{0, a\}$, $z$ and $y$ are determined, and so $X$ is determined. 

In other words, the total number of $X$ in this case is $(q-1)(q-2)+(q-1)$

\item Suppose we are in the second case, i.e. $a=x$ and $y=b$. Then, in this situation, for each $z\ne c$, $t$ is determined. This means that there are only $q-1$ matrices $X$ in this case. 
\end{itemize}
Putting these cases together, we obtain that the number of matrices $X\in SL_2(\mathbb{F}_q)$ such that $X\ne A$ and $\det(A-X)=0$ is $q^2-1$.

Hence,  $G_{11}$ is a regular graph of degree $d_{11}=q^3-q^2-q \sim q^3$. 
\end{proof}
We now observe that $G_{11}$ is a connected graph. Indeed, it has been shown in \cite[Proposition $3.9$]{Yesim} that any matrix in $M_2(\mathbb{F}_q)$ can be written as a sum of two matrices in $SL_2(\mathbb{F}_q)$. Thus, for any $A, B\in SL_2(\mathbb{F}_q)$, 
 we can assume that 
\[B-A=C_1+C_2,\]
for $C_1, C_2\in SL_2(\mathbb{F}_q)$.  

Hence, for any two vertices $A$ and $B$ in $V_{11}$, there is always a path of length two between them, namely, $A, A+C_1, B$. So the graph $G_{11}$ is connected. 

On the other hand,  one can check that $G_{11}$ is the complement of the graph $G_{31}$ (see Subsection \ref{subsection 2.2.3}) and so  its second largest eigenvalue is bounded by $q^{3/2}$ (see \cite[Lemma $8.5.1$]{alge} for more details). In short, we have the following lemma.  
\bigskip
\begin{lemma}
$G_{11}$ is a connected graph, and it is an
\[(q^3-q, q^3-q^2-q, c_{11}q^{3/2})-\mbox{graph},\]
for some constant $c_{11}$.
\end{lemma}

Let $G_{12}=(V_{12}, E_{12})$ be the graph defined as follows:

\[V_{12}=M_2(\mathbb{F}_q), (A, B)\in E_{12} \Leftrightarrow A-B\in GL_2(\mathbb{F}_q).\]

It has been proved in \cite{Yesim} that $G_{12}$ is a connected graph, and it is an 
\[(q^4,~ |GL_2(\mathbb{F}_q)|\sim q^4,~ q^2)-\mbox{graph}.\]

For two graphs $G_1=(V_1, E_1)$ and $G_2=(V_2, E_2)$, the tensor product $G_1\otimes G_2$ is a graph with vertex set $V(G_1\otimes G_2)=V_1\times V_2$, and there is an edge between $(u,v)$ and $(u', v')$ if and only if $(u, u')\in E_1$ and $(v, v')\in E_2$. Suppose that the adjacency matrices of $G_1$ and $G_2$ are $A$ and $B$, respectively. Then the adjacency matrix of $G_1\otimes G_2$ is the tensor product of $A$ and $B$. It is well-known that if $\gamma_1, \ldots, \gamma_n$ are eigenvalues of $A$ and $\gamma_1', \ldots, \gamma_m'$ are eigenvalues of $B$, then the eigenvalues of $A\otimes B$ are $\gamma_i\gamma_j'$ with $1\le i \le n$, $1\le j\le m$ (see \cite{maris} for more details).

Observe that if  $G'_1$ denotes the tensor product of $G_{11}$ and $G_{12},$ then we have the following lemma. 
\begin{lemma}\label{lmx1}
 $G'_{1}$ is a connected graph, and it is an 
\[(q^7-q^5, \sim q^7, c_{11} q^{11/2})-\mbox{graph}.\]
\end{lemma}
\subsubsection{The eigenvalues of $M_{2i}$}\label{Subsection2.2.2}
We can view $G_2$ as the union of $(q-1)$ graphs $G_{2i}$ with $i\in \mathbb{F}_q\setminus\{0\}$, where $G_{2i}=(V_{2i}, E_{2i})$ is defined as follows:
\[V_{2i}=SL_2(\mathbb{F}_q)\times M_2(\mathbb{F}_q),\]
and there is an edge between  $(A_1, C_1)$ and $(A_2, C_2)$ if $\det(A_1-A_2)=i$ and $\det(C_1-C_2)=i$. Without loss of generality, we may assume that $i=1$. 

Let $G_{211}$ be the \textit{special unit Cayley graph} $\Gamma(M_2(\mathbb{F}_q), SL_2(\mathbb{F}_q))$ defined as follows:
\[V_{211}=M_2(\mathbb{F}_q),\]
and there is an edge between $A$ and $B$ if $A-B\in SL_2(\mathbb{F}_q)$. As in the Section $2$, we have the following lemma on the $(n, d, \lambda)$ form of the graph $\Gamma(M_2(\mathbb{F}_q), SL_2(\mathbb{F}_q))$. 
\bigskip
\begin{lemma}\label{useful} The special unit Cayley graph $G_{211}$ is a connected graph, and it is an 
\[(q^4, \sim q^3, 2q^{3/2})-\mbox{graph}.\]
\end{lemma}
It is interesting to note that the graph $G_{211}$ has diameter two. We refer readers to \cite{Yesim} for a detailed proof. 

Let $G_{212}$ be a graph defined as follows:
\[V_{212}=SL_2(\mathbb{F}_q),\]
and there is an edge between $A$ and $B$ if $A-B\in SL_2(\mathbb{F}_q)$. Using elementary calculations, we can prove that $G_{212}$ is a regular graph of degree $d_{212}\sim q^2$, and it is a connected graph.

To bound the second largest eigenvalue of this graph, we need the interlacing eigenvalue theorem. 
\bigskip
\begin{theorem}[\textbf{Interlacing eigenvalue}, \cite{alge}]\label{thmz}
Let $A$ be an $n\times n$ matrix with eigenvalues $\lambda_1\le\cdots\le \lambda_n$. Let $B$ be an $m\times m$ symmetric minor of $A$ with eigenvalues $\mu_1\le \cdots\le \mu_{m}$. Then 
\[\lambda_i\le \mu_i\le \lambda_{i+n-m}.\]
\end{theorem}

Let $M_{21j}$ be the adjacency matrix of $G_{21j}$ for $j=1,2.$ It is clear that $M_{212}$ is a symmetric minor of $M_{211}$. Suppose $\mu_1\le \cdots\le \mu_m =d_{212}\sim q^2$ are eigenvalues of $M_{212}$ with $m=q^{3}-q$. The second largest eigenvalue of $M_{212}$ is bounded by $\max\{|\mu_1|, |\mu_{m-1}|.\}$ 

Let $\lambda_1, \ldots, \lambda_n$ denote eigenvalues of $M_{211}$ where $\lambda_1\le \cdots \le \lambda_n\sim q^3$ with $n=q^4.$ Theorem \ref{thmz} tells us that 
\[\lambda_1\le \mu_1, ~ \mu_{m-1}\le \lambda_{n-1}.\]

Thus we see that $\max\{|\mu_1|, |\mu_{m-1}|\}$ is bounded by the second largest eigenvalue of $M_{211}$ which is at most $2q^{3/2}$. In conclusion, it follows that $G_{212}$ is an 
\[(q^3-q, \sim q^2, 2q^{3/2})-\mbox{graph}.\]

It follows from the definition of $G_{21}$ that $G_{21}$ is the tensor product of $G_{211}$ and $G_{212}$.  Hence we have the following lemma on the $(n, d, \lambda)$ form of $G_{21}$.
\bigskip
\begin{lemma}\label{lmx2}
The connected graph $G_{21}$ is an
\[(q^7-q^5, \sim q^5, c_{21}q^{9/2})-\mbox{graph},\]
for some positive constant $c_{21}$.
\end{lemma}

\subsubsection{The eigenvalues of $M_3$} \label{subsection 2.2.3}
Let $G_{31}=(V_{31}, E_{31})$ be a graph defined as follows:
\[V_{31}=SL_2(\mathbb{F}_q), ~(A, B)\in E_{31}\Leftrightarrow \det(A-B)=0.\]
It has been shown above that $G_{31}$ is a regular graph of degree $q^2-1$. 
\bigskip
\begin{lemma}
The graph $G_{31}$ is a connected graph, and it is an
\[(q^3-q, q^2-1, c_{31}q^{3/2})-\mbox{graph},\]
for some positive constant $c_{31}$.
\end{lemma}
\begin{proof}
To prove this lemma, we need to count the number of common neighbors of two vertices. 
Let 
\[A_1=\begin{pmatrix}
a&b\\
c&d\\
\end{pmatrix} \ne A_2=\begin{pmatrix}
a'&b'\\
c'&d'\\
\end{pmatrix}.\]
We now count the number of matrices $X\in SL_2(\mathbb{F}_q)$ of the following form

\[
X=\begin{pmatrix}
x&y\\
z&t\\
\end{pmatrix}\]
such that $\rank(A_1-X)=\rank(A_2-X)=1$. 

This implies that there exist $\lambda_1, \lambda_2\in \mathbb{F}_q$ such that 
\begin{equation}\label{eq41}c-z=\lambda_1(a-x), ~d-t=\lambda_1(b-y), ~~c'-z=\lambda_2(a'-x), ~d'-t=\lambda_2(b'-y).\end{equation}
Hence, we obtain 
\begin{equation}\label{eq42}c-c'=\lambda_1a-\lambda_2a'+x(\lambda_2-\lambda_1), ~d-d'=\lambda_1b-\lambda_2b'+y(\lambda_2-\lambda_1).\end{equation}
Since $xt-yz=1$, from the equations (\ref{eq41}), we have 
\begin{equation}\label{eq43}
\lambda_1(ay-bx)=1-xd+yc, ~\lambda_2(a'y-b'x)=1-xd'+yc'.
\end{equation}
To count the number of matrices $X$, our main strategy is to count the number pairs $(\lambda_1, \lambda_2)$ satisfying (\ref{eq42}) and (\ref{eq43}), and then for each pair of $\lambda_1, \lambda_2$, we count the number of tuples $(x, y, z, t)$ satisfying (\ref{eq41})-(\ref{eq43}). We now fall into two cases: 

{\bf Case $1$:} In this case, we will find conditions on $A_1$ and $A_2$ such that there are pairs $(\lambda_1, \lambda_2)$ with $\lambda_1=\lambda_2$. 

\begin{itemize}
\item[I.] If we have $a\ne a'$, $b\ne b'$, and $\frac{c-c'}{a-a'}=\frac{d-d'}{b-b'}$, then we have 
\[\lambda_1=\lambda_2=\frac{c-c'}{a-a'}=\frac{d-d'}{b-b'}.\]
Moreover we also have
\[\lambda_1=\lambda_2=\frac{1+b'c-a'd}{ab'-ba'}.\]
Indeed, from the equation 
\[\frac{c-c'}{a-a'}=\frac{1+b'c-a'd}{ab'-ba'},\]
we obtain 
\[a'(2+bc'+b'c-a'd)=a(b'c'+1).\]
On the other hand, it follows from the condition $\frac{c-c'}{a-a'}=\frac{d-d'}{b-b'}$ that 
\[bc'+b'c-a'd=cb-ad+c'b'-a'd'+ad'=-2+ad'.\]
Substituting this to the above equation gives us 
\[a(b'c'+1)=a'(ad'),\]
which is always true since $b'c'-a'd'=-1$. 

It is clear that in this case we have $\det(A_1-A_2)=0$.

\item[II.] If we have $a\ne a'$, $b=b',$ then $d=d'$, and
\[\lambda=\frac{c-c'}{a-a'}=\frac{1+b'c-a'd}{ab'-ba'}, ~\det(A_1-A_2)=0.\]
\item[III.] If $a=a', b\ne b'$, then $c=c'$, and we have 
\[\lambda=\frac{d-d'}{b-b'}=\frac{1+b'c-a'd}{ab'-ba'}, ~\det(A_1-A_2)=0.\]
\item[IV.] If $a=a'$, $b=b'$, and $c\ne c'$ or $d\ne d'$, then we have 
\[\det(A_1-A_2)=0, \]
but there is no pair $(\lambda_1, \lambda_2)$ with $\lambda_1=\lambda_2$. 

For other cases, there is no pair $(\lambda_1, \lambda_2)$ with $\lambda_1=\lambda_2$.
\end{itemize}
For each pair $(\lambda_1, \lambda_2)$ with $\lambda_1=\lambda_2$, by using some elementary calculations, one can show that there are only $q$ common neighbors of $A_1$ and $A_2$. Note that if $\det(A_1-A_2)\ne 0$, then there is no pair $(\lambda_1, \lambda_2)$ with $\lambda_1=\lambda_2$.

{\bf Case $2$:} In this case, we will find conditions on $A_1$ and $A_2$ such that there are pairs $(\lambda_1, \lambda_2)$ with $\lambda_1\ne \lambda_2$. 

It follows from (\ref{eq42}) that 
\[x=\frac{c-c'+\lambda_2a'-\lambda_1 a}{\lambda_2-\lambda_1}, ~y=\frac{d-d'+\lambda_2b'-\lambda_1b}{\lambda_2-\lambda_1}.\]
Thus the equation (\ref{eq43}) gives us 
\begin{equation}\label{eq44}\lambda_2(1+b'c-a'd-\lambda_1(ab'-ba'))=(bc'+1-ad')\lambda_1+d'c-c'd.\end{equation}
We now consider the following cases:
\begin{itemize}
\item[I.] If we have $ab'-ba'\ne 0$ and $(a-a')(d-d')=(b-b')(c-c')$, then we see that

if $\lambda_1=\frac{1+b'c-a'd}{ab'-ba'}$, then any $\lambda_2$ will satisfy (\ref{eq44}). Thus the number of pairs $(\lambda_1, \lambda_2)$ with $\lambda_1\ne \lambda_2$ is $q-1$. 

For $\lambda_1\ne\frac{1+b'c-a'd}{ab'-ba'}$,  $\lambda_2$ is determined by (\ref{eq44}). 

Therefore, in total, the number of pairs $(\lambda_1, \lambda_2)$ with $\lambda_1\ne \lambda_2$ is $2q-2$. 
\item[II.]
If $ab'-ba'=0$, then we have 
\[\lambda_2(1+b'c-a'd)+\lambda_1(ad'-bc'-1)=cd'-dc'.\]

Since $A_1\ne A_2$, we have at least one of the terms $(1+b'c-a'd), (ad'-bc'-1), cd'-dc'$ is non-zero (To see this: if all these three terms are zero and $ab'-ba'=0$, we then have $A_1[b,-a]^T=[0,-1]^T$ and $A_1[b',-a']^T=[0,-1]^T$. Taking inverse matrix of $A_1$ shows $b=b', a=a'$).

If either $1+b'c-a'd\ne 0$ or $ad'-bc'-1\ne 0$, then we can choose $\lambda_2$ (or $\lambda_1$) arbitrarily, and so $\lambda_1$ (or $\lambda_2$) is determined. In this case we have the number of pairs $(\lambda_1, \lambda_2)$ with $\lambda_1\ne \lambda_2$ is $q$. 

If $1+b'c-a'd= 0$, $ad'-bc'-1= 0$, and $cd'-dc'\ne 0$, then there is no pair $(\lambda_1, \lambda_2)$ with $\lambda_1\ne \lambda_2$. 
\item[III.] If we have $ab'-ba'\ne 0$ and $(a-a')(d-d')\ne (c-c')(b-b')$, then 
\[\lambda_1=\frac{1+b'c-a'd}{ab'-ba'}\]
is not a solution of (\ref{eq44}). 

Thus for each $\lambda_1\ne \frac{1+b'c-a'd}{ab'-ba'}$, $\lambda_2$ is  uniquely determined. 

Since $\det(A_1-A_2)\ne 0$, we have $\lambda_2\ne \lambda_1$. 

In this case, the number of pairs $(\lambda_1, \lambda_2)$ with $\lambda_1\ne \lambda_2$ is $q-1$. 

We note that for each pair $(\lambda_1, \lambda_2)$ with $\lambda_1\ne \lambda_2$, then the matrix $X$ is  uniquely determined. 
\end{itemize}

In summary, the number of common neighbors of two vertices $A_1\ne A_2$ can be stated as follows:

\begin{itemize}
\item If $\det(A_1-A_2)=0$, then we have 
\begin{enumerate}
\item If $a\ne a'$, $b\ne b'$, $ab'-ba'\ne 0$, then there is only one pair $(\lambda_1, \lambda_2)$ with $\lambda_1=\lambda_2=(c-c')/(a-a')$, and there are $(q-1)$ pairs $(\lambda_1, \lambda_2)$ with $\lambda_1\ne \lambda_2$. So the number of $X$ is $2q-1$.
\item If $a\ne a'$, $b=b'$, $d=d'$, $ab'-ba'\ne 0$, then 
there is only one pair $(\lambda_1, \lambda_2)$ with $\lambda_1=\lambda_2=(c-c')/(a-a')$, and there are $(q-1)$ pairs $(\lambda_1, \lambda_2)$ with $\lambda_1\ne \lambda_2$. So the number of $X$ is $2q-1$. 
\item If $a=a'$, $b\ne b'$, $c=c'$, then there is only one pair $(\lambda_1, \lambda_2)$ with $\lambda_1=\lambda_2=(d-d')/(b-b')$, and there are $(q-1)$ pairs $(\lambda_1, \lambda_2)$ with $\lambda_1\ne \lambda_2$. So the number of $X$ is $2q-1$.
\item If $a=a'$, $b=b'$, $c\ne c'$ or $d\ne d'$, then we have $cd'-dc'\ne 0$ since $A_1\ne A_2$. This implies that there is no solution $(\lambda_1, \lambda_2)$. So the number of $X$ is $0$. 
\end{enumerate}
\item If $\det(A_1-A_2)\ne 0$, then we have 
\begin{enumerate}
\item If $ab'-ba'\ne 0$, there is no pair $(\lambda_1, \lambda_2)$ with $\lambda_1=\lambda_2$, and there are $(q-1)$ pairs $(\lambda_1, \lambda_2)$ with $\lambda_1\ne \lambda_2$. So the number of $X$ is $q-1$.
\item If $ab'-ba'=0$, then we have either $1+b'c-a'd\ne 0$ or $a'd-bc'-1\ne 0$ since $\det(A_1-A_2)\ne 0$. There is no pair $(\lambda_1, \lambda_2)$ with $\lambda_1=\lambda_2$, and there are $q$ pairs $(\lambda_1, \lambda_2)$ with $\lambda_1\ne \lambda_2$. So the number of $X$ is $q$.
\end{enumerate}
\end{itemize}

These cases tell us that the graph $G_{31}$ is a connected graph. 

Let $M_{31}$ be the adjacency matrix of $G_{31}$. Then $M_{31}$ can be presented as follows:

\[M_{31}^2=(q^2-q+1)I+(q-1)J+(q)E_{31}+(q)E_{32}+qE_{33}-(q-1)E_{34}+E_{35},\]
where 

$I$ is the identity matrix, 

$J$ is the all-one matrix, 

$E_{31}$ is the adjacency matrix of the graph $G_{31}=(V_{31}, E_{31})$ defined as follows: 
\[V_{31}=SL_2(\mathbb{F}_q),\]
and there is an edge between $A_1$ and $A_2$ if $a\ne a'$, $b\ne b'$, $ab'-ba'\ne 0$, and $\det(A_1-A_2)=0$. 

$E_{32}$ is the adjacency matrix of the graph $G_{32}=(V_{32}, E_{32})$ defined as follows: 
\[V_{32}=SL_2(\mathbb{F}_q),\]
and there is an edge between $A_1$ and $A_2$ if $a\ne a'$, $b= b'$, $d=d'$, $ab'-ba'\ne 0$, and $\det(A_1-A_2)=0$. 

$E_{33}$ is the adjacency matrix of the graph $G_{33}=(V_{33}, E_{33})$ defined as follows: 
\[V_{33}=SL_2(\mathbb{F}_q),\]
and there is an edge between $A_1$ and $A_2$ if $a= a'$, $b\ne b'$, $c=c'$, and $\det(A_1-A_2)=0$.

$E_{34}$ is the adjacency matrix of the graph $G_{34}=(V_{34}, E_{34})$ defined as follows: 
\[V_{34}=SL_2(\mathbb{F}_q),\]
and there is an edge between $A_1$ and $A_2$ if $a= a'$, $b= b'$, either $c\ne c'$ or $d\ne d'$, and $\det(A_1-A_2)=0$. 

$E_{35}$ is the adjacency matrix of the graph $G_{35}=(V_{35}, E_{35})$ defined as follows: 
\[V_{35}=SL_2(\mathbb{F}_q),\]
and there is an edge between $A_1$ and $A_2$ if $ab'-ba'=0$ and either $1+b'c-a'd\ne 0$ or $a'd-bc'-1\ne 0$, and $\det(A_1-A_2)\ne 0$.

Since $G_{31}$ is a regular graph of order $q^2-1$, its largest eigenvalue is $q^2-1$. Suppose $\lambda_2$ is the second largest eigenvalue of $M_{31}$ and $v_{\lambda_2}$ is the corresponding eigenvector. Then we have 
\[\lambda_2^2 v_{\lambda_2}=(q^2-q+1)v_{\lambda_2}+\left(qE_{31}+qE_{32}+qE_{33}-(q-1)E_{34}+E_{35}\right)v_{\lambda_2}.\]

Thus $v_{\lambda_2}$ is an eigenvector of 
\[(q^2-q+1)I+\left(qE_{31}+qE_{32}+qE_{33}-(q-1)E_{34}+E_{35}\right),\]
and $\lambda_2^2$ is the corresponding eigenvalue. 

Since the graphs $E_{3i}$ are regular graphs of degree at most $q^2$, and eigenvalues of a sum of matrices are bounded by the sum of the largest eigenvalues
of the summands, we get 
\[\lambda_2\ll q^{3/2},\]
which concludes the proof of the theorem.
\end{proof}












Since the graph $G_3$ is the tensor product of $G_{12}$ and $G_{31}$,  we have the following lemma. 
\bigskip
\begin{lemma}\label{lmx3} The graph $G_3$ is a connected graph, and it is an
\[(q^7, \sim q^6, c_3q^{11/2})-\mbox{graph},\]
for some positive constant $c_3$.
\end{lemma}
\subsubsection{The eigenvalues of $M_4$}
Let $G_{41}=(V_{41}, E_{41})$ be the graph defined as follows:
\[V_{41}=M_2(\mathbb{F}_q), ~(A, B)\in E_{41}\Leftrightarrow \det(A-B)=0.\]
It is easy to see that $G_{41}$ is a regular graph with the degree $\sim q^3$. It is clear that this graph is the complement of the graph $G_{12}$, and so the second largest eigenvalue is bounded by $q^2$. One can use a similar argument as in the previous subsection to show that  $G_{41}$ is connected.

In other words, the graph $G_{41}$ is a connected graph, and it is an
\[(q^4, \sim q^3, c_{41}q^{2})-\mbox{graph},\]
for some positive constant $c_{41}$.

One can check that $G_4$ is the tensor product graph of $G_{11}$ and $G_{41}$. Therefore we have the following result on the $(n, d, \lambda)$ form of $G_{4}$.
\bigskip
\begin{lemma}\label{lmx4} The graph $G_4$ is a connected graph, and it is an
\[(q^7, q^6, c_4q^{11/2})-\mbox{graph},\]
for some positive constant $c_4$.
\end{lemma}
\subsubsection{The eigenvalues of  of $M_5$, $M_6$, $M_7$, and $M_8$}\label{lmx5}
It follows from the definitions of $M_5$ and $M_6$ that their eigenvalues are bounded by $q^3$. This is enough for our purpose.

For the graph $G_7$, it is clear that for each $A\in SL_2(\mathbb{F}_q)$, the number of $C\in SL_2(\mathbb{F}_q)$ such that $\rank(A-C)=1$ is $q^2-1$, and for each matrix $C$, $\alpha$ is determined. 

For each $\alpha$, there are $q^2$ matrices $D\in M_2(\mathbb{F}_q)$ of the following form 
\[D=\begin{pmatrix}
u&v\\
\alpha u&\alpha v\\
\end{pmatrix}\]
such that 
$\rank(C-D)=1$. 

In conclusion, the degree of each vertex in $G_7$ is bounded by $q^4$. This is enough for our purpose. 

Similarly, one can prove that $G_8$ is a regular graph of degree $q^5$. This satisfies our purpose.

We are now ready to prove Theorem \ref{thm2}.
\begin{proof}[Proof of Theorem \ref{thm2}]
Suppose $M_2$ is the adjacency matrix of $G_2$. We have proved that 
\begin{equation}\label{eqx99}
M_2M_2^t=(|SL_2(\mathbb{F}_q)|-1)I+J-M_1+\sum_{i\in \mathbb{F}_q\setminus \{0\}}M_{2i}-M_3-M_4-M_5-M_6+(q-1)M_7-M_8.\end{equation}

Note that all the $M_j$ and $M_{2i}$ that appear in (\ref{eqx99}) are adjacency matrices of regular graphs. Thus all of them share an eigenvector $v_1=(1, \ldots, 1)^T$, and the corresponding eigenvalues are their degrees. 

Suppose that $\lambda_2$ is the second largest eigenvalue of $M$ and $v_2$ is the corresponding eigenvector. Then we have $J\cdot v_2=\mathbf{0}$. 

The equation (\ref{eqx99}) gives us 
\[(\lambda_2^2-q^3+q+1)v_2=(-M_1+\sum_{i\in \mathbb{F}_q\setminus \{0\}}M_{2i}-M_3-M_4-M_5-M_6+(q-1)M_7-M_8)\cdot v_2.\]

Therefore, $v_2$ is an eigenvector of the sum 
\begin{equation}\label{eqqqx}-M_1+\sum_{i\in \mathbb{F}_q\setminus \{0\}}M_{2i}-M_3-M_4-M_5-M_6+(q-1)M_7-M_8,\end{equation}
and $\lambda_2^2-q^3+q+1$ is its second largest eigenvalue. 

Let $w_2, \ldots, w_{q^{7}}$ be orthogonal vectors in $\mathbb{F}_q^{q^7}$ such that $\{w_1:=v_1, w_2, \ldots, w_{q^7}\}$ form a normal basis. Let $P$ be the matrix such that the $i$-th column is $w_i$. 

Then we have 
\[P^{-1}M_jP=\begin{bmatrix}
    \mbox{degree ($G_j$)} & * & * & \dots  & *& * \\
    0 & * & * & \dots  & * & *\\
    0 & * & * & *  & *& * \\
    \vdots & \vdots & \vdots & \vdots & \vdots & \vdots \\
    0 & * & * & \dots  & *& *\\
 0 & * & * & \dots  & *& *\\
\end{bmatrix},\]
and 
\[P^{-1}M_{2i}P=\begin{bmatrix}
    \mbox{degree ($G_{2i}$)} & * & * & \dots  & *& * \\
    0 & * & * & \dots  & * & *\\
    0 & * & * & *  & *& * \\
    \vdots & \vdots & \vdots & \vdots & \vdots & \vdots \\
    0 & * & * & \dots  & *& *\\
 0 & * & * & \dots  & *& *\\
\end{bmatrix},\]
for $1\le j\le 8$, $1\le i \le q$.

We have 
\[P^{-1}(-M_1+\sum_{i\in \mathbb{F}_q\setminus \{0\}}M_{2i}-M_3-M_4-M_5-M_6+(q-1)M_7-M_8)P=\]
\[\begin{bmatrix}
    \sum_{i=1}^q\mbox{degree ($G_{2i}$)}+\sum_{j=1, j \ne 2}^{8} \mbox{degree ($G_j$)} & * & * & \dots  & *& * \\
    0 & * & * & \dots  & * & *\\
    0 & * & * & *  & *& * \\
    \vdots & \vdots & \vdots & \vdots & \vdots & \vdots \\
    0 & * & * & \dots  & *& *\\
 0 & * & * & \dots  & *& *\\
\end{bmatrix},\]
which implies that the largest eigenvalue of the sum (\ref{eqqqx}) is 
\[\sum_{i=1}^q\mbox{degree ($G_{2i}$)}+\sum_{j=1, j \ne 2}^{8} \mbox{degree ($G_j$)}.\]
Notice that the sets of eigenvalues of two matrices $M$ and $P^{-1}MP$ are the same.

Furthermore,  since the graphs $G_1, G_{2j}, G_3, G_4$ are connected graphs,  their largest eigenvalues have multiplicity one.  Therefore, the second largest eigenvalues of the sum (\ref{eqqqx}) is bounded by 
\[\sum_{i=1}^q\lambda_2(G_{2i})+\sum_{j=1, 3, 4}\lambda_2(G_j)+\lambda_1(G_5)+\lambda_1(G_6)+\lambda_1(G_7)+\lambda_1(G_8),\]
where $\lambda_1(G_j)$ is the largest eigenvalue of $G_j$, and $\lambda_2(G_j)$ is the second largest eigenvalue of $G_j$. 

Applying Lemmas \ref{lmx1}-\ref{lmx4}, and results of Subsection \ref{lmx5}, we obtain 
\[\lambda_2\ll q^{11/4},\]
which completes the proof of the theorem.
\end{proof}
\paragraph{Proof of Theorem \ref{thm-sum9}:} The proof of Theorem \ref{thm-sum9} is identical with that of Theorem \ref{thm-sum7} except that we use Theorem \ref{thm2} instead of Theorem \ref{thm1}. Thus we leave remaining details to the reader. $\square$
\section{Discussions}
One might ask if it is possible to prove Corollary \ref{nice9} (Corollary \ref{nice19}) using the approach in the proof of Theorem \ref{thm-sum4''''} (Theorem \ref{thm-sum2''''}). To this end,  we need to  define the following version of the sum-product digraph over $SL_2(\mathbb{F}_q)$. More precisely, let $G=(A\cup B, E)$ be a bipartite graph with $A=SL_2(\mathbb{F}_q)\times SL_2(\mathbb{F}_q)$ and $B=SL_2(\mathbb{F}_q)\times M_2(\mathbb{F}_q)$, there is an edge between $(X, Y)\in A$ and $(Z, T)\in B$ if $XZ=Y+T$. By a direct computation, we have $|A|\sim q^6$ and $|B|\sim q^7$. It is not hard to check that each vertex in $A$ is of degree $\sim q^3$, and each vertex in $B$ is of degree $\sim q^2$. This implies that $\lambda_1=-\lambda_n=q^{5/2}$ with $n=|A|+|B|$.

The main difficult problem is to bound the third largest eigenvalue of $G$. The problem becomes much harder than the proof of Theorem \ref{thm1}, for instance, in the graph $G_2$, we have relaxed conditions of vertices in $A$, namely, $Y\in M_2(\mathbb{F}_q)$ instead of $SL_2(\mathbb{F}_q)$, but the proof of Theorem \ref{thm2} on the $(n, d, \lambda)$ form of $G_2$ is complicated and technical. Giving a good upper bound of the third eigenvalue of $G$ is outside the realm of methods in this paper. 

For $A\subset SL_2(\mathbb{F}_q)$, it follows from Corollary \ref{Adidaphat-09} and Corollary \ref{Adidaphat-10} that the two polynomials $f=x+yz$ and $f=x(y+z)$ satisfy $|f(A, A, A)|\gg q^4$. It would be interesting to study polynomials of other forms, for example, $f=(x-y)^2+z$. 


\section*{Acknowledgments}
D. Koh was supported by Korea National Science Foundation grant NRF-2018R1D1A1B07044469. T. Pham was supported by Swiss National Science Foundation grant P2ELP2175050. C-Y. Shen was supported in part by MOST, through grant 104-2628-M-002-015 -MY4. The authors are grateful to the referee for useful comments and suggestions.

\end{document}